\documentclass{IEEEtran4PSCC}

\usepackage{cite}

\ifCLASSINFOpdf
   \usepackage[pdftex]{graphicx}
  \graphicspath{{../fig/}}
  \DeclareGraphicsExtensions{.pdf,.jpeg,.png}
\else
   \usepackage[dvips]{graphicx}
  \graphicspath{{../eps/}}
  \DeclareGraphicsExtensions{.eps}
\fi
\usepackage[cmex10]{amsmath}
\usepackage{array}

\ifCLASSOPTIONcompsoc
 \usepackage[caption=false,font=normalsize,labelfont=sf,textfont=sf]{subfig}
\else
 \usepackage[caption=false,font=footnotesize]{subfig}
\fi
\usepackage{url}

\usepackage{xcolor}
\usepackage{amsfonts} 
\usepackage{amssymb} 
\usepackage{amsthm}
\usepackage{eurosym}
\usepackage{mathtools}
\DeclareMathOperator*{\argmin}{arg\,min}
\theoremstyle{definition}
\newtheorem{theorem}{Theorem}
\usepackage{threeparttable}
\usepackage{longtable}
\usepackage{balance}

\makeatletter
\let\old@ps@headings\ps@headings
\let\old@ps@IEEEtitlepagestyle\ps@IEEEtitlepagestyle
\def\psccfooter#1{%
    \def\ps@headings{%
        \old@ps@headings%
        \def\@oddfoot{\strut\hfill#1\hfill\strut}%
        \def\@evenfoot{\strut\hfill#1\hfill\strut}%
    }%
    \def\ps@IEEEtitlepagestyle{%
        \old@ps@IEEEtitlepagestyle%
        \def\@oddfoot{\strut\hfill#1\hfill\strut}%
        \def\@evenfoot{\strut\hfill#1\hfill\strut}%
    }%
    \ps@headings%
}
\makeatother

\begin{document}
\title{Computing Necessary Conditions for\\ Near-Optimality in Capacity Expansion\\ Planning Problems}




\author{\IEEEauthorblockN{Antoine Dubois\IEEEauthorrefmark{1} and
Damien Ernst\IEEEauthorrefmark{1}\IEEEauthorrefmark{2}}
\IEEEauthorblockA{\IEEEauthorrefmark{1} Department of Electrical Engineering and Computer Science, ULi\`ege, Liège, Belgium}
\IEEEauthorblockA{\IEEEauthorrefmark{2} LTCI, Telecom Paris, Institut Polytechnique de Paris, Paris, France}
}

\maketitle

\begin{abstract}

In power systems, large-scale optimisation problems are extensively used to plan for capacity expansion at the supra-national level. 
However, their cost-optimal solutions are often not exploitable by decision-makers who are preferably looking for features of solutions that can accommodate their different requirements. 
This paper proposes a generic framework for addressing this problem. 
It is based on the concept of the epsilon-optimal feasible space of a given optimisation problem and the identification of necessary conditions over this space.
This framework has been developed in a generic case, and an approach for solving this problem is subsequently described for a specific case where conditions are constrained sums of variables. 
The approach is tested on a case study about capacity expansion planning of the European electricity network to determine necessary conditions on the minimal investments in transmission, storage and generation capacity.

\end{abstract}

\begin{IEEEkeywords}
Capacity expansion planning, decision-making, epsilon-optimality, necessary conditions, optimisation
\end{IEEEkeywords}

\thanksto{\noindent Antoine Dubois is a Research Fellow of the F.R.S.-FNRS, of which he acknowledges the financial support.}

\section{Introduction} \label{sec:intro}

In the coming decades, the European power system will have to face the challenges related to the integration of massive amounts of renewable energy sources and a high level of electrification of the heating, transport and industrial sectors.

The size and level of integration of the European electricity network (\textit{i.e.}, at the transmission level, thousands of substations and power lines connecting them) entail a level of complexity in planning this transition that requires using detailed optimisation models. 
The increased sophistication of these models comes with drawbacks.
In particular, these models essentially focus on unique cost-based optimal solutions that are too restrictive and do not encompass the different requirements of many stakeholders intervening in the decision process for new investments in capacity.

In our opinion, it is preferable to provide \textit{necessary conditions} in capacity investments that guarantee a constrained suboptimality and provide a common ground over which decision-makers can settle and create solutions that accommodate their needs.
For example, we could compute the minimum required investment in transmission lines per country to ensure a maximum deviation of 10\% from the optimum.
Alternatively, one might be interested in knowing if a particular technology - for example, Li-Ion battery or some renewable energy source (RES) type - is necessary for a cost-efficient energy transition.\\

In this paper, a framework is presented to derive necessary conditions for $\epsilon$-optimality and applied to a capacity expansion planning problem.
In Section \ref{sec:literature-review}, we discuss the literature related to the optimisation concepts that underlie the framework. 
The optimisation framework itself is presented in Section \ref{sec:problem-formulation}. 
Section \ref{sec:method} specifies this framework to the case of conditions consisting of constrained sums of variables and provides a fully-defined methodology for computing \textit{non-implied necessary conditions} in such a context.
This methodology is afterwards illustrated on an expansion planning problem in Section \ref{sec:test-case}.
Section \ref{sec:conclusion} concludes with the description of future research directions.
Finally, Appendix \ref{app:network-model} gathers more detailed data on the modelling of the network used in the test case.\\

\section{Literature review} \label{sec:literature-review}

Decision-making based on optimisation results is a complex task.
Indeed, this exercise lies at the frontier between human intelligence and machine power whose coupling is challenging \cite{brill1990mga}, sometimes referred to as post-normal science \cite{ravetz1999post}. 
Decision-making is linked to \textit{deep uncertainties} \cite{yue2018review} where, among other topics, the desirability of alternative outcomes corresponding to different policy objectives is subject to disagreement among stakeholders.
These uncertainties and disagreements imply that relying on a single cost-based optimum is often not sufficient. 
Indeed, there is no guarantee that the findings obtained via this optimum will be robust regarding parameter perturbation, nor that they will satisfy conflicting objectives. 
Moreover, as shown in \cite{trutnevyte2016does}, cost-optimal scenarios are not adequate to approximate real-world problems, such as those encountered in the context of the energy transition.
This problem highlights the need for the ``role of optimisation model methods to be re-thought in full recognition of these limitations'', as suggested in \cite{brill1979use}. 
Those authors advocate optimisation methods that ``should be used to generate planning alternatives, facilitate their evaluation and elaboration, provide insights and serve as catalysts for human creativity''.

We argue that such a \textit{re-thinking} can be achieved by orienting the use of optimisation methods in the search for conditions that are respected across multiple feasible solutions and guarantee a constrained level of suboptimality.
The advantage of this approach over unique cost-optimal solutions is to provide decision elements that all stakeholders can agree on and built on using their creativity.

Those solutions can be obtained in a variety of ways. 
One possibility is the use of multi-objective (or multi-criteria) optimisation \cite{ehrgott2005multicriteria}.
In this field, one searches not for a single solution but a set of \textit{efficient} or \textit{Pareto-optimal} solutions, translating some trade-offs between objectives. 
A notable drawback with these methods is the general assumption of knowing which objectives are at stake in the problem and being able to model in some form those objectives. 
Objectives that are a priori unknown (either because of lack of knowledge or unconscious biases) or non-modellable are, thus, left apart.

A technique that circumvents this limitation is what some authors refer to as \textit{``Modeling to Generate Alternatives''} \cite{brill1982modeling}. 
It consists of exploring solutions located in the inferior or suboptimal region of an optimisation problem \cite{decarolis2011using}.
The underlying motivation of this approach is that this region might contain solutions that are better in terms of some unmodelled objectives. 
Several authors, such as \cite{PRICE2017356}, \cite{li2017investment}, \cite{nacken2019integrated},  \cite{sasse2019distributional} or \cite{NEUMANN2021106690}, exploit this technique. 
However, their main objective is to show the variety of solutions that can be extracted rather than to systematically compute conditions that are respected by those solutions.
In this paper, we present a framework that puts the identification of such conditions at the centre of the optimisation process.

Finally, in the domain of multi-objective optimisation, \cite{bandaru2017data} have surveyed advanced data-driven methods for extracting \textit{commonalities} among Pareto-optimal solutions. 
Our framework aims at providing the ground for developing such techniques in suboptimal spaces.\\

\section{Problem formulation} \label{sec:problem-formulation}

Let us consider the following optimisation problem
\begin{align} \label{equ:original-problem}
\min_{x\in\mathcal{X}} f(x)
\end{align}
with $\mathcal{X}$ being the feasible space and $f:\mathcal{X} \rightarrow \mathbb{R}^+$ the objective function. 
Let $x^*$ be an optimal solution. 
We define an \textit{$\epsilon$-optimal space} as follows:
\begin{align*} \label{equ:epsilon-space}
\mathcal{X}^\epsilon = \{x \in \mathcal{X} \mid f(x) \leq (1+\epsilon) f(x^*), \epsilon \geq 0\}.
\end{align*}
The set $\mathcal{X}^\epsilon$, depicted in Figure~\ref{fig:epsilon_optimal_space}, contains only the feasible solutions with an objective value no greater than $(1+\epsilon)f(x^*)$. 
We define $\epsilon$ as the \textit{suboptimality coefficient} of such a space, \textit{i.e.} specifying by how much the objective values of the solutions in the space deviate at most from the optimal objective value.
\begin{figure}[!ht]
    \centering
    \includegraphics[width=2.3in]{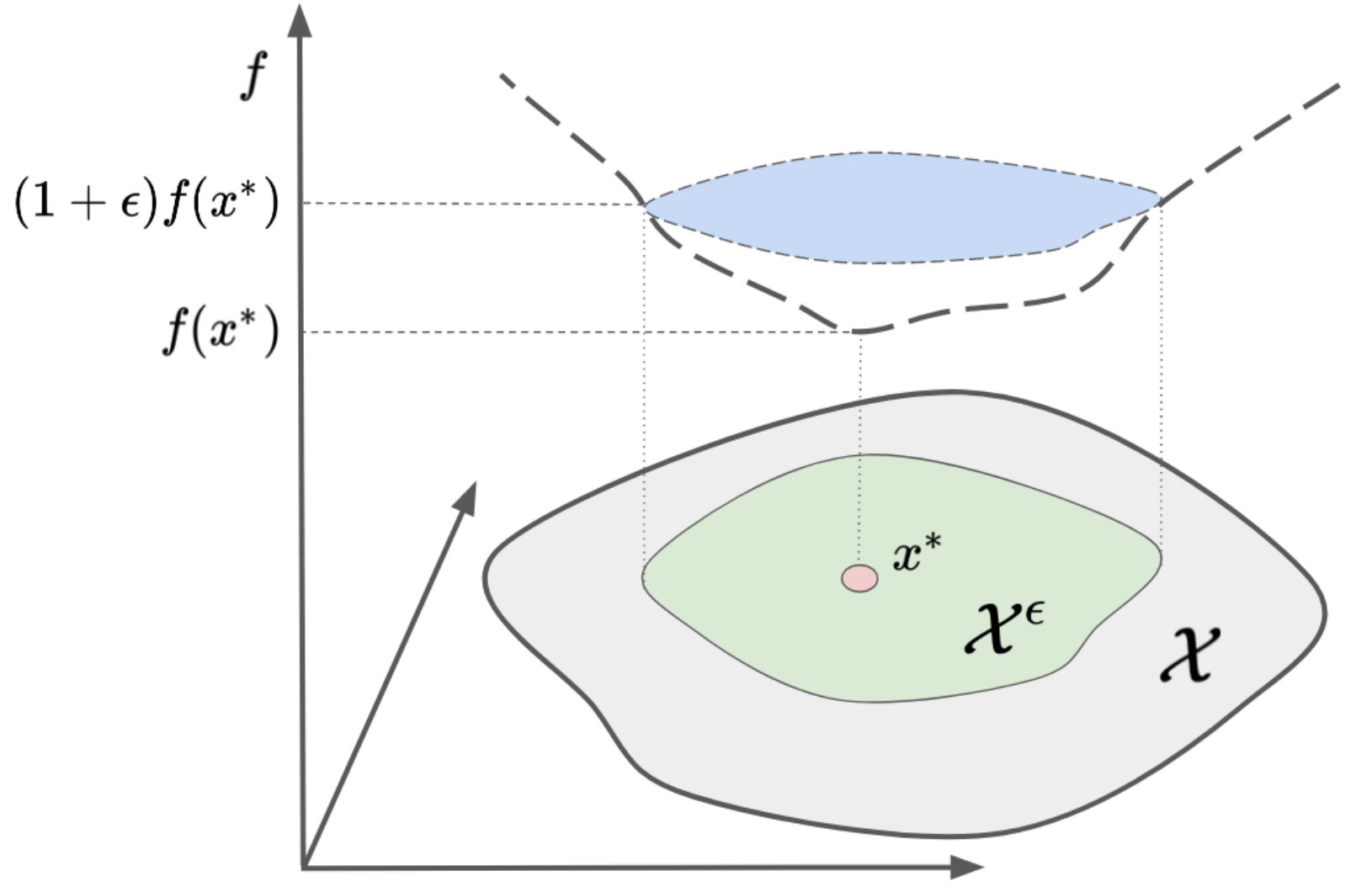}
	\caption{Three-dimensional representation of an $\epsilon$-optimal space. In the horizontal plane, the space of feasible solutions $\mathcal{X} \subset \mathbb{R}^2$ is depicted while the vertical axis represents the objective function. The red dot is the optimal solution $x^*$ corresponding to the minimal value of the objective function over the feasible space. The blue space coincides with the hyper-plane $f(x) = (1+\epsilon) f(x^*)$ allowing one to determine the $\epsilon$-optimal space $\mathcal{X}^\epsilon$ shown in green.}
	\label{fig:epsilon_optimal_space}
\end{figure}

Let us define \textit{conditions} as functions $\phi: \mathcal{X} \rightarrow \{0, 1\}$.
Our goal is to identify, among a set $\Phi$ of conditions, the ones which are true for any solutions in $\mathcal{X}^\epsilon$. 
These conditions are called \textit{necessary conditions} for $\epsilon$-optimality, where the parameter $\epsilon$ allows one to monitor the level of suboptimality of those necessary conditions.
Mathematically, 
$$\Phi^{\mathcal{X}^\epsilon} = \{\phi \in \Phi \mid \forall x \in \mathcal{X}^\epsilon: \phi(x)=1\}$$
is the set of necessary conditions for a given feasible space $\mathcal{X}$, sets of conditions $\Phi$ and suboptimality coefficient $\epsilon$.

\subsection{Non-implied necessary conditions on sets of parametric conditions}

The goal of the methodology presented in this paper is to support decision-makers in their decision process.
However, as explained in the next paragraph, even for a single set of conditions, an infinite number of necessary conditions can be derived.
Such quantity of information can not be used efficiently to take decisions.
In this section, the concept of \textit{non-implied necessary condition} is introduced as a solution to this problem. 

Let consider the feasible space $\mathcal{X} = \mathbb{R}$ and and a set of parametric conditions of the type
$$\Phi = \{\phi_c(x) \coloneqq x \geq c | c \in \mathbb{R}\}.$$
This set contains an infinite number of conditions and can lead to identifying an infinite number of necessary conditions, with which decision-makers might find it cumbersome to deal.

For instance, let $\phi_1(x) \coloneqq x > 1$ be a necessary condition for $\epsilon$-optimality (\textit{i.e.} $\forall x \in \mathcal{X}^\epsilon: \phi_1(x) = 1$).
This automatically \textit{implies} that all $\phi_c$ where $c < 1$ are necessary conditions.
Indeed $\forall x \in \mathcal{X}^\epsilon: x > 1 \Rightarrow x > c$.
The only condition that cannot be implied to be a necessary condition from the knowledge of other necessary conditions is the necessary condition $\phi_c$ with the largest value of $c$.

This necessary condition is what constitutes a \textit{non-implied necessary condition}.
This is a condition that cannot be \textit{implied} to be a necessary condition from the sole knowledge of other conditions that constitute necessary conditions.
To minimise the number of necessary conditions that need to be identified and presented to decision-makers, the focus should be placed on the identification of non-implied necessary conditions.

The notion of implication can be formalised by defining the space over which a condition $\phi$ is true,
$$\mathcal{I}_\phi = \{x\in\mathcal{X} \mid \phi(x)=1\}.$$ 
A condition $\phi_2$ implies $\phi_1$ if $\mathcal{I}_{\phi_2} \subset \mathcal{I}_{\phi_1}$, \textit{i.e.} $\phi_1$ is true for all $x \in \mathcal{X}$ over which $\phi_2$ is true.
Considering sets of parametric conditions, $\mathcal{I}_{\phi_2} = \mathcal{I}_{\phi_1}$ happens only if both conditions are equal. 
Using this notion, conditions can be defined to be necessary conditions if the space over which they are true encloses $\mathcal{X}^\epsilon$, and non-implied necessary conditions if this space does not include any of the spaces over which other necessary conditions are true. 
Mathematically, the set of non-implied necessary conditions for $\epsilon$-optimality is defined as
\begin{align*}
    \overline{\Phi}^{\mathcal{X}^\epsilon} = \{\phi \in \Phi^{\mathcal{X}^\epsilon} \mid \forall \phi' \in \Phi^{\mathcal{X}^\epsilon}\setminus\{\phi\}: \mathcal{I}_{\phi'} \not\subset \mathcal{I}_{\phi}\}.
\end{align*}

Figure~\ref{fig:true-spaces-1} provides an illustration of two necessary conditions $\phi_1$ and $\phi_2$, with $\phi_2$ implying $\phi_1$.
Considering a set of conditions $\Phi = \{\phi_1, \phi_2\}$ containing uniquely $\phi_1$ and $\phi_2$, the set of necessary conditions is given by $\Phi^{\mathcal{X}^\epsilon} = \{\phi_1, \phi_2\}$, and the set of non-implied necessary conditions by $\overline{\Phi}^{\mathcal{X}^\epsilon} = \{\phi_2\}$. 
Figure~\ref{fig:true-spaces-2} provides an illustration of two other necessary conditions $\phi_3$ and $\phi_4$, with no implication. 
Considering a set of conditions $\Phi = \{\phi_3, \phi_4\}$, then $\overline{\Phi}^{\mathcal{X}^\epsilon}  = \Phi^{\mathcal{X}^\epsilon} = \Phi$.\\ 
\begin{figure}[!t]
    \centering
    \subfloat[]{\includegraphics[width=1.5in]{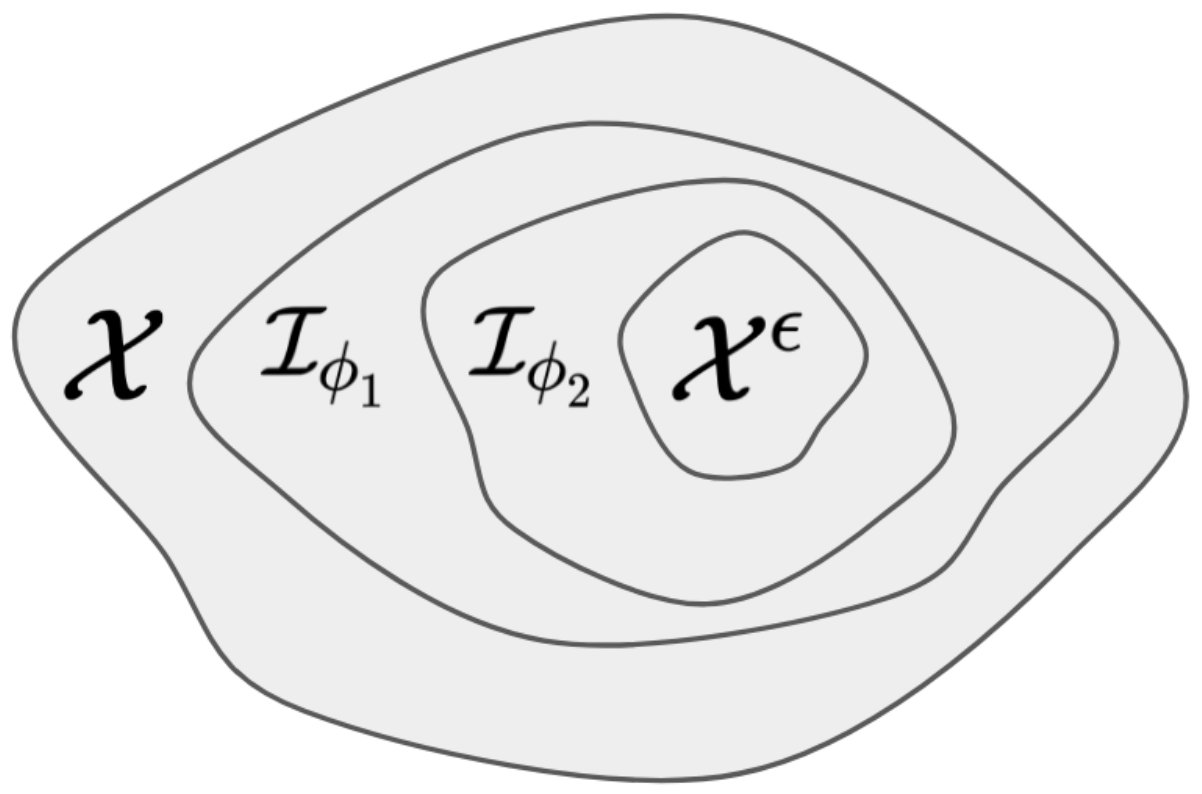}%
    \label{fig:true-spaces-1}}
    \hfil
    \subfloat[]{\includegraphics[width=1.5in]{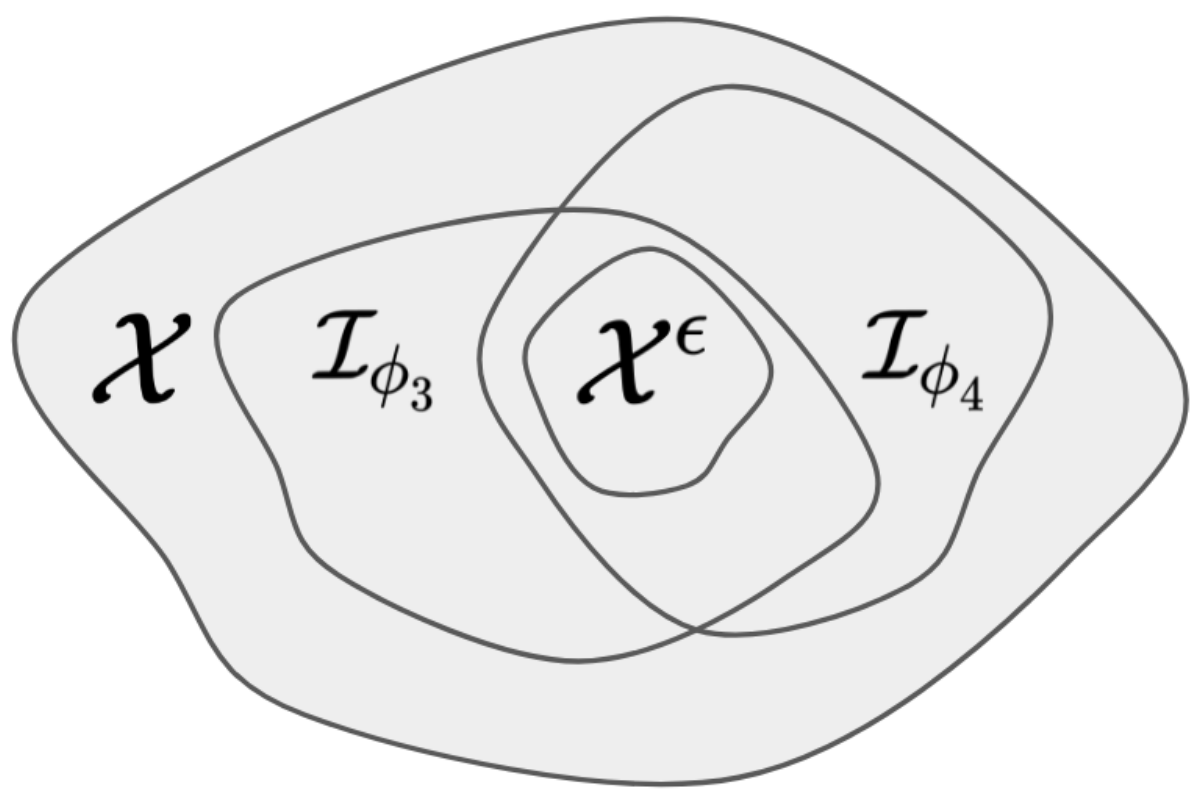}%
    \label{fig:true-spaces-2}}
    \caption{Graphical illustration of implication using spaces over which conditions are true.}
    \label{fig:true-spaces}
\end{figure}

\section{Proposed method} \label{sec:method}

In this section, a specific instance of the generic problem specified in Section \ref{sec:problem-formulation} is considered. 
This instance is characterised by conditions consisting of parametric constrained sums of variables. 
We show how for each value of the parameter defining these sums, a unique non-implied necessary condition can be determined.\\

\begin{theorem}
Let $\mathcal{X} \subset \mathbb{R}^n$, $f:\mathcal{X} \rightarrow \mathbb{R}^+$ and
\begin{align*} 
    \Phi_{\mathbf{d}} = \{&\phi_{\mathbf{d}}^ c(\mathbf{x}) \coloneqq \mathbf{d}^T\mathbf{x} \geq c \mid c \in \mathbb{R}\},
\end{align*}
where $\mathbf{x} \in \mathcal{X}$, be a set of conditions consisting of constrained sums of variables $\mathbf{d}^T\mathbf{x}=\sum_{i=1}^n d_ix_i$ defined by $\mathbf{d} \in \{0, 1\}^n$.
Let $c^* = \min_{\mathbf{x}\in\mathcal{X}^\epsilon} \mathbf{d}^T\mathbf{x}$ then
$$\phi_\mathbf{d}^{c^*}\coloneqq \mathbf{d}^T\mathbf{x} \geq c^*$$
is the only element in the set of non-implied necessary conditions $\overline{\Phi}_{\mathbf{d}}^{\mathcal{X}^\epsilon}$.\\

\end{theorem}

\begin{proof}

Let us first show that the set of necessary conditions is equal to $$\Phi_\mathbf{d}^{\mathcal{X}^\epsilon} = \{\phi_\mathbf{d}^c | c \leq c^*\}.$$
By definition, 
$$c^* = \min_{\mathbf{x}\in\mathcal{X}^\epsilon} \mathbf{d}^T\mathbf{x}$$
is the smallest value that $\mathbf{d}^T\mathbf{x}$ can take over $\mathcal{X}^\epsilon$. 
This implies that
$$\phi_\mathbf{d}^{c^*}(\mathbf{x}) \coloneqq \mathbf{d}^T\mathbf{x} \geq c^*$$
is true for all $\mathbf{x} \in \mathcal{X}^\epsilon$. 
Similarly, if $c < c^*$, we know that
$$\mathbf{d}^T\mathbf{x} \geq c^* > c$$
is true for all $\mathbf{x}\in\mathcal{X}^\epsilon$. 
Thus, all conditions $\phi_\mathbf{d}^c$ such that $c \leq c^*$ are necessary conditions.
For $c > c^*$ however, at the optimum $\mathbf{x}^*_\epsilon = \argmin_{\mathcal{X}^\epsilon} \mathbf{d}^T\mathbf{x}$, we have
$$\mathbf{d}^T\mathbf{x}^*_\epsilon = c^* < c$$ 
which implies that the condition 
$$\phi_\mathbf{d}^c(\mathbf{x}) \coloneqq \mathbf{d}^T\mathbf{x} > c$$
is not true for all $\mathbf{x}$ in $\mathcal{X}^\epsilon$,  as $\mathbf{x}^*_\epsilon \in \mathcal{X}^\epsilon$.
Therefore, all conditions $\phi_\mathbf{d}^c$ such that $c > c^*$ are not necessary conditions.\\

\noindent Now let us prove
$$\overline{\Phi}_\mathbf{d}^{\mathcal{X}^\epsilon} = \{\phi_\mathbf{d}^{c^*}\}.$$
This means that all $\phi_\mathbf{d}^c$ with $c < c^*$ are implied by and do not imply $\phi_\mathbf{d}^{c^*}$. 
This can be shown by proving that, for any $c < c^*$,
$$\mathcal{I}_{\phi_\mathbf{d}^{c^*}} \subset \mathcal{I}_{\phi_\mathbf{d}^c} \text{ and } \mathcal{I}_{\phi_\mathbf{d}^{c}} \not\subset \mathcal{I}_{\phi_\mathbf{d}^{c^*}}.$$
We have $\mathcal{I}_{\phi_\mathbf{d}^{c^*}} \subset \mathcal{I}_{\phi_\mathbf{d}^c}$ because, as shown before, for any $\mathbf{x}$, if $\phi_\mathbf{d}^{c^*}(\mathbf{x})$ is true, $\phi_\mathbf{d}^c(\mathbf{x})$ with $c<c^*$ is also true. 
Moreover, $\mathcal{I}_{\phi_\mathbf{d}^{c}} \not\subset \mathcal{I}_{\phi_\mathbf{d}^{c^*}}$. 
Indeed, the element $\mathbf{x}$ such that 
$$\mathbf{d}^T\mathbf{x} = c$$ 
is an element of $\mathcal{I}_{\phi_\mathbf{d}^c}$ but not of $\mathcal{I}_{\phi_\mathbf{d}^{c^*}}$.

\end{proof}

\section{Test case} \label{sec:test-case}

This methodology will now be applied to a specific test case.
The test case is articulated around the problem of capacity expansion planning of the European electricity grid within the objective of the European Union to be carbon-neutral by 2050.
Typically, the objective of this problem is to determine capacity investments in transmission, generation and storage assets as well as operation of those assets to satisfy electrical demand while minimising capital and marginal costs.

Decision-makers might be interested in knowing the necessary conditions on the required amount of capacity to be invested in each of those technologies at the European and national levels to ensure that they do not experience more than a well-specified level of cost-suboptimality.

Our methodology will be applied to this problem for computing non-implied necessary conditions for achieving $\epsilon$-optimality on five technologies.
More specifically, required minimum investments are first computed for groups of lines at the European, national and individual-line levels.
Then, necessary conditions are determined for storage and RES generation, including onshore wind, offshore wind, and utility-scale PV, over the whole network.

In the following section, a short contextualisation of the test case is presented. 
The test case is then defined following the terms of the methodology presented above. 
It is followed by a short analysis of the optimal solution of the expansion planning problem before describing necessary conditions.

\subsection{Context} \label{sec:test-case-context}

The geographical scope of the expansion planning problem is set to Europe.
All countries in the European continent are included, except for Russia, Iceland and some small countries such as Cyprus, Malta and Liechtenstein. 
In this problem, the network is represented as a grid made up of nodes and lines.
When applied at the European level, nodes are generally clustered by country while lines correspond to aggregations of pre-existing or planned transmission lines between these countries. 
Figure~\ref{fig:topologies-a} shows the nodes and lines forming the network. 
In addition, generators and storage devices are attached to each of those nodes.
The temporal scope of the problem is set to one full year, corresponding to the year 2050. 
More details on the modelling of the network can be found in Appendix~\ref{app:network-model}.

\subsection{Optimisation problem} \label{sec:test-case-optimisation}

The expansion planning problem is solved using linear optimisation via the open-source tool PyPSA \cite{PyPSA}. 
In this context, the elements composing problem \eqref{equ:original-problem} are described briefly below.\\

\noindent {\bf Objective function $f$.}
The objective of the problem is to minimise the total annual system cost. 
To be more specific, the objective function $f$ corresponds to the sum of annualised capital fixed costs and variable costs of generation, storage and transmission across the network.\\

\noindent {\bf Feasible space $\mathcal{X}$.}
%
The feasible space can be modelled as $\mathcal{X} = \{\mathbf{x} \in \mathbb{R}^n | A\mathbf{x} \geq \mathbf{b}, A \in \mathbb{R}^{m\times n}, \mathbf{b}\in \mathbb{R}^m\}$ with $m\in \mathbb{N}$ and $n \in \mathbb{N}_0$.
The variables $\mathbf{x}$ correspond to investment (\textit{i.e.} how much capacity must be added where and to what technology) and operational variables (\textit{e.g.} which quantity of energy each generator must produce at each time step).
All variables are continuous, as investments are continuous and unit commitment is not modelled. 
The bounds on those variables are composed via technical and physical constraints modelled as linear constraints.
In addition, a constraint imposing a 99\% reduction on CO2 emissions compared to 1990 levels is added.
This value is set Europe-wise and is not set to 100\% to ensure the feasibility of the problem.\\

\begin{figure}[!t]
    \centering
    \subfloat[Initial capacity.]{\includegraphics[width=2.2in]{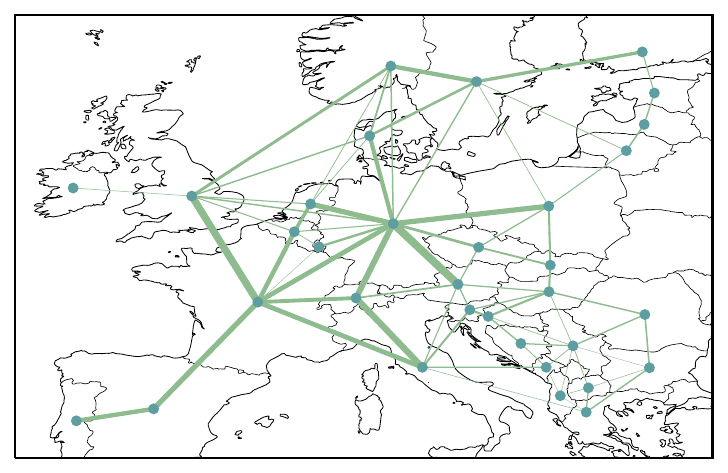}%
    \label{fig:topologies-a}}
    \hfil
    \subfloat[Additional capacity to be cost-optimal.]{\includegraphics[width=2.2in]{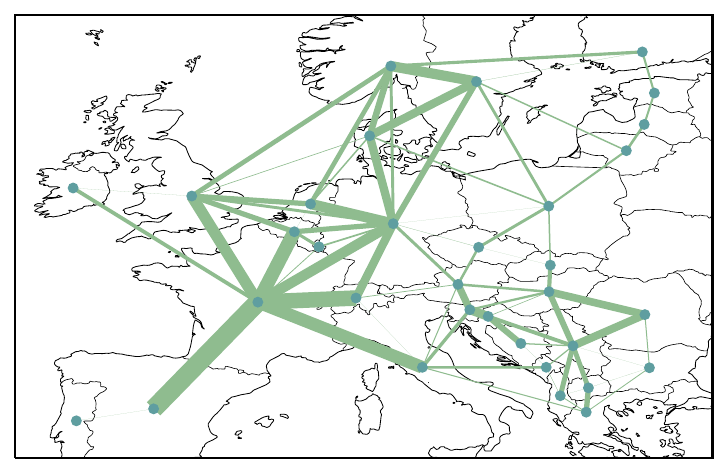}%
    \label{fig:topologies-b}}
    \caption{European electricity grid representation with the width of lines proportional to capacity in GW.}
    \label{fig:topologies}
\end{figure}

\noindent {\bf Set of conditions $\Phi$.}
As in Section \ref{sec:method}, to compute non-implied necessary conditions corresponding to minimum capacity investments, sets of conditions of form 
\begin{align*}
   \Phi_{\mathbf{d}} = \{\mathbf{d}^T\mathbf{x}_I\geq c\mid c \in \mathbb{R}\}
\end{align*}
are used, with $\mathbf{x}_I$ a vector of size $|I|$ that collects the different investment variables and $\mathbf{d} \in \{0, 1\}^{|I|}$.
Depending on the type of investment variables for which a non-implied necessary condition is desired, it suffices to define an appropriate $\mathbf{d}$, \textit{i.e.} whose elements corresponding to the required variables are set to 1, and others to 0.\\
%

\noindent {\bf Suboptimality coefficient $\epsilon$.}
As mentioned above, necessary conditions are valid for a given value of the suboptimality coefficient $\epsilon$. 
In this study, necessary conditions are computed for different values of $\epsilon$ - ranging from 0\% (\textit{i.e.} optimality) to 20\% - to see how minimal capacity investment evolves with the suboptimality coefficient.\\

\noindent {\bf Computation of non-implied necessary conditions.}
As a reminder, the computation consists in the following steps:
\begin{enumerate}
    \item Compute an optimal solution $\mathbf{x}^*$ for problem \eqref{equ:original-problem}.
    \item For a given suboptimal coefficient $\epsilon$, compute an $\epsilon$-optimal space $\mathcal{X}^\epsilon$ using this solution.
    \item For this $\epsilon$ and a value of $\mathbf{d}$, extract a non-implied necessary condition by solving $\min_{\mathbf{x} \in \mathcal{X}^\epsilon} \mathbf{d}^T\mathbf{x}$.
    \item Repeat step (2) and (3) to obtain non-implied necessary conditions for different values of $\epsilon$ and $\mathbf{d}$.
\end{enumerate}

\subsection{Optimal solution.}

\begin{table*}[!ht]
\renewcommand{\arraystretch}{1.3}
\centering
\caption{Optimal Capacities.}
\label{tab:optimal-capacities}
\begin{tabular}{|c|c|c||c|c|c||c|c||c|}
    \hline
    \multicolumn{3}{|l||}{TWkm} & \multicolumn{6}{l|}{GW}\\
    \hline
    AC & DC & AC+DC & Onshore wind & Offshore wind & Utility PV & CCGT & OCGT & Li-Ion\\
    \hline
    128 & 90 & 218 & 168 & 327 & 367 & 49 & 0 & 249\\
    \hline
\end{tabular}
\end{table*}

\begin{figure*}[!t]
    \centering
    \subfloat[Sum of the capacities of all lines.]{\includegraphics[height=1.65in]{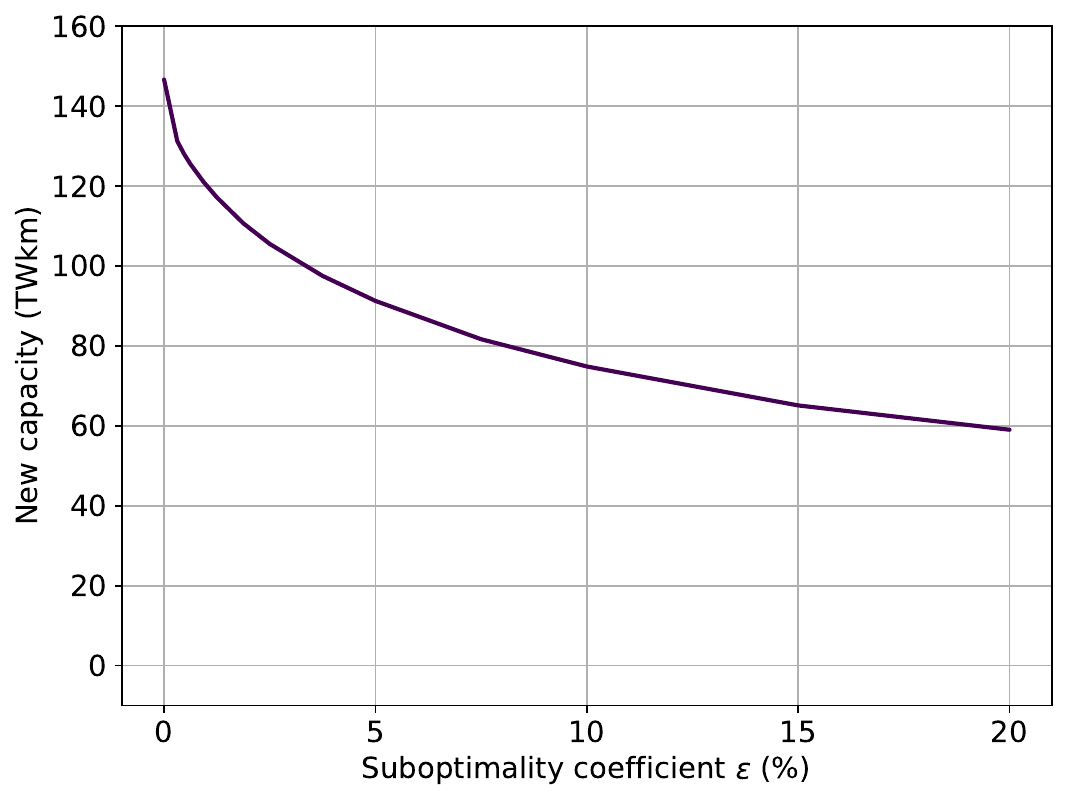}
    \label{fig:whole-cap}}
    \hfil
    \subfloat[Sum of the capacities of country lines.]{\includegraphics[height=1.65in]{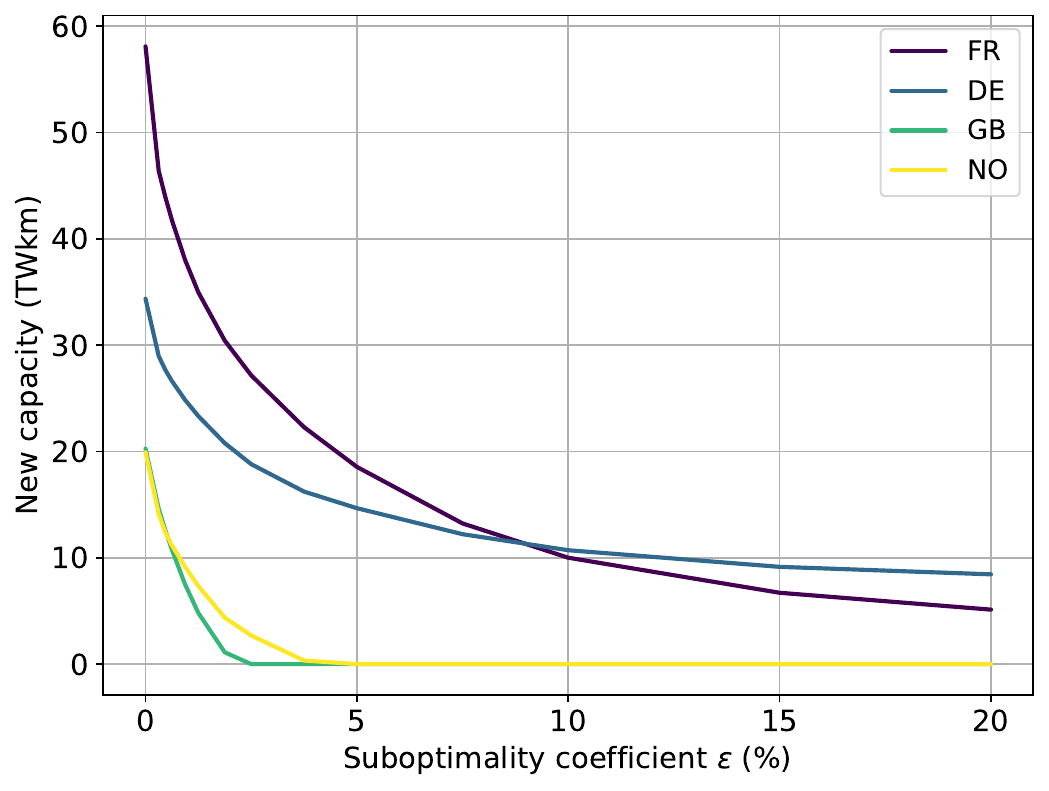}%
    \label{fig:countries-cap}}
    \hfil
    \subfloat[Capacity of individual lines.]{\includegraphics[height=1.65in]{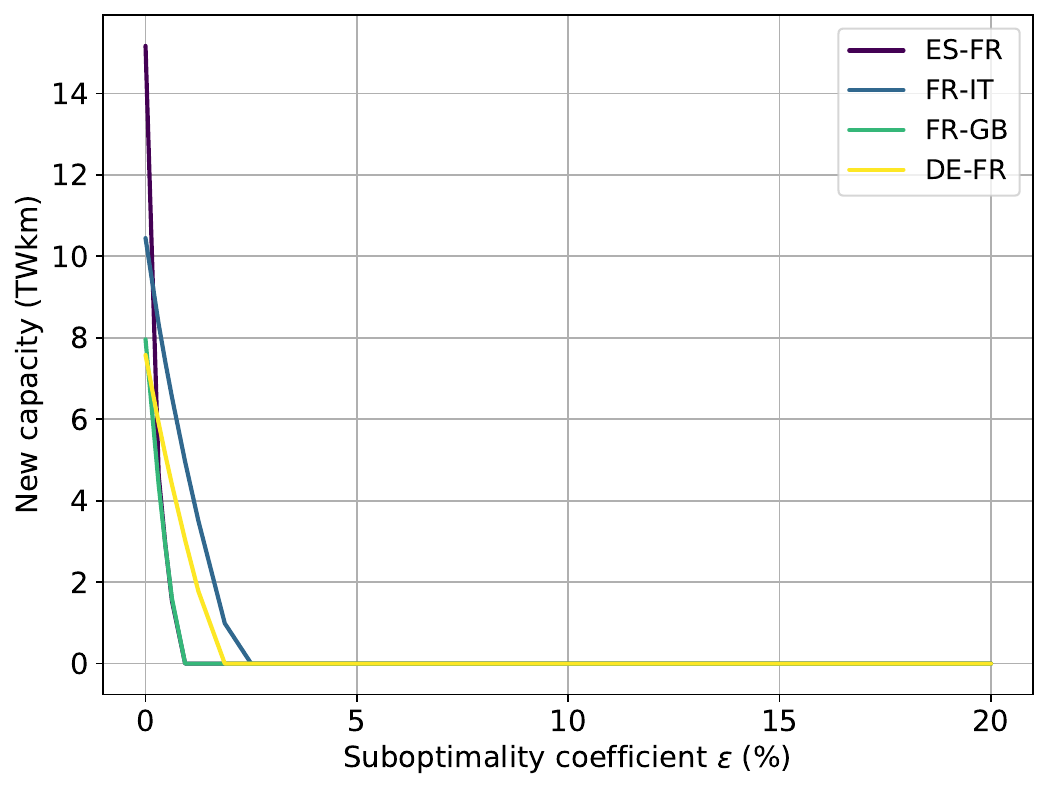}%
    \label{fig:lines-cap}}
    \hfil\\
    \subfloat[Sum Li-Ion batteries capacities.]{\includegraphics[height=1.65in]{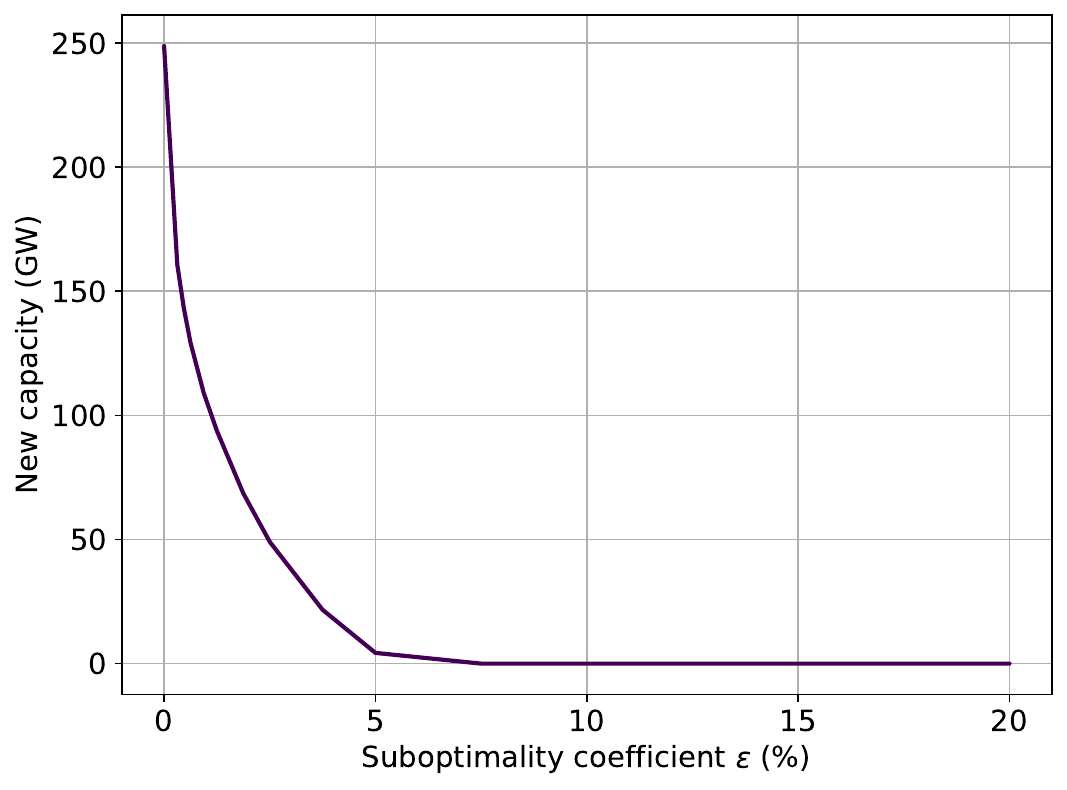}%
	\label{fig:storage-cap}}
	\hfil
    \subfloat[Sum of renewable energy generators capacities (onshore wind, offshore wind, utility PV and sum of the three).]{\includegraphics[height=1.65in]{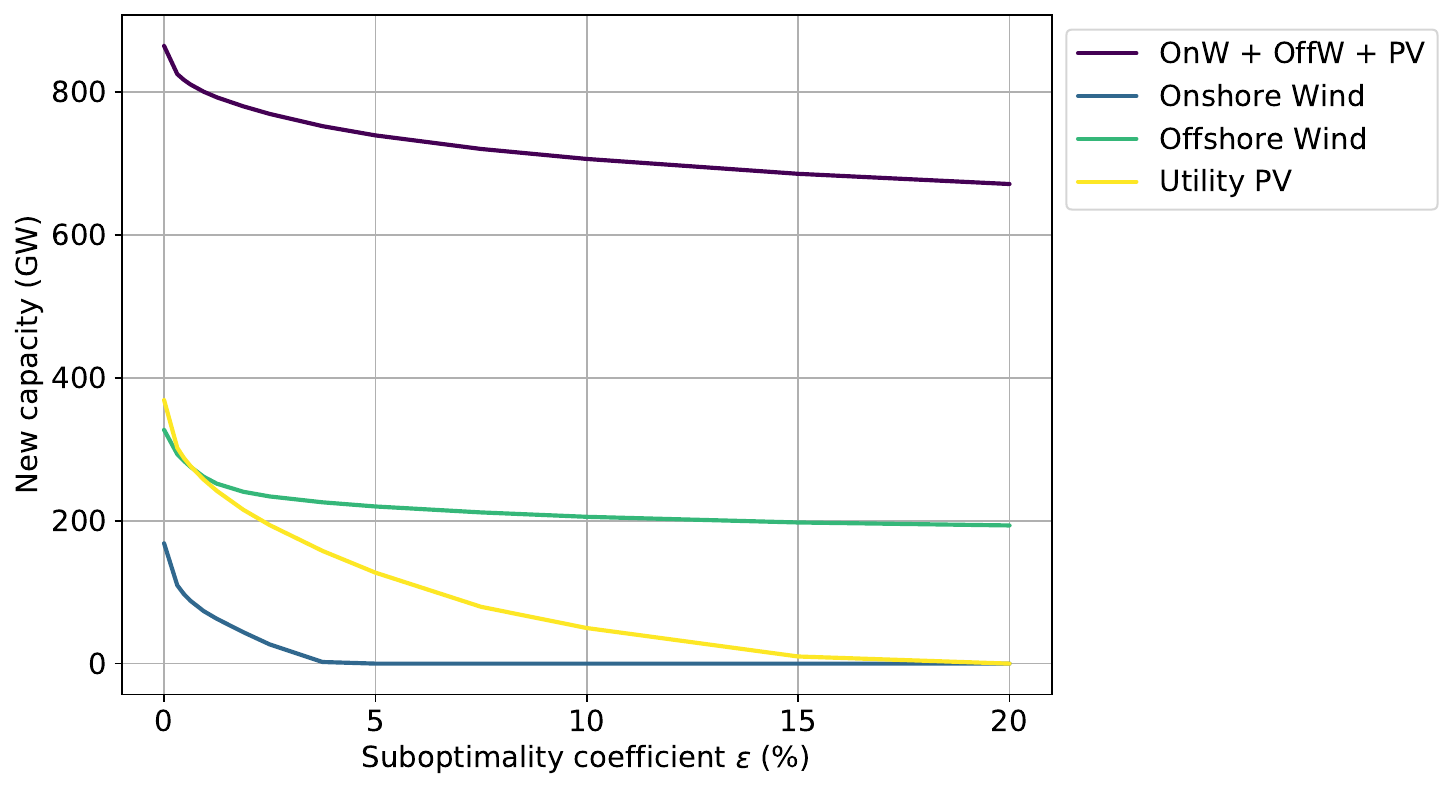}
	\label{fig:generation-cap}}
    \caption{Non-implied necessary conditions on new capacity for different levels of suboptimality.}
    \label{fig:transmission-cap}
\end{figure*}

Figure~\ref{fig:topologies} shows the initial topology - with the widths of lines proportional to their capacity in GW - and how much capacity should be added to be optimal (\textit{i.e.} obtained from $\mathbf{x}^*$).
Table~\ref{tab:optimal-capacities} lists the optimal capacities for the technologies that are expanded in the problem.
Capacities of lines are expressed in TWkm, \textit{i.e.} the power capacity installed over a given distance.

For transmission, with the addition of 146 TWkm of lines, the initial 72 TWkm is tripled.
Capacity is added to almost all lines, but major investments are made in some countries like Germany and France.
In these particular cases, the sums of the optimal capacities of the lines connected to the country are equal to 79 TWkm (from an initial capacity of 21 TWkm) and 57 TWkm (from 23 TWkm).

The capacities of RES, including onshore wind, offshore wind and utility PV, are massive, reaching 862 GW across the three technologies.
Gas plants are also deployed though on a smaller scale due to the constraint on CO2 emissions.
Finally, a substantial amount of Li-Ion batteries is built all around Europe with the main hub being in Spain where a lot of PV is also deployed.
Note that the initial capacity of these technologies (\textit{i.e.}, RES, storage and gas) are null.

From this optimal solution, the $\epsilon$-optimal spaces corresponding to the values of $\epsilon$ listed in the previous section are defined and necessary conditions for transmission, Li-Ion storage and RES generation are computed.

\subsection{Necessary conditions on transmission capacity}

In this case, the variables of interest correspond to transmission capacity variables, \textit{i.e.} how much capacity - in TWkm - should be added to each line.
We minimise capacity in TWkm (\textit{i.e.} the power capacity of the line multiplied by its length) because this value is a good representative of both the physical and economic investments in transmission assets.
Necessary conditions are first computed for the sum of capacities of all lines in the network. 
The analysis is then refined by looking at the sum of capacities of lines connected to a given country and, finally, at the capacities of unique lines.\\

\noindent {\bf Whole network.}
The first non-implied necessary conditions  to be $\epsilon$-optimal consist of the required minimum capacity to be added to the whole network.
These are obtained by setting to one all values of $\mathbf{d}$ corresponding to transmission capacity investment variables. The other values of $\mathbf{d}$ are set to zero.

Figure~\ref{fig:whole-cap} shows the values of the required new capacity to sustain a certain $\epsilon$-optimality for different values of $\epsilon$. 
The required new capacity drops rapidly for small values of $\epsilon$ and, for a suboptimality coefficient $\epsilon$ of 10\%, it has already been divided by a factor two compared to the optimum.
However, for larger $\epsilon$, the decrease slows down, and the necessary conditions start reaching a slowly decreasing plateau around 60 TWkm.

\noindent {\bf National lines.}
These first results show how much capacity should be added to the initial network to keep costs above a certain $\epsilon$-optimality.
National transmission system operators could also be interested in the minimum capacity required to connect their country to the rest of the system.
It can be identified by setting all values in $\mathbf{d}$ to zero except for those corresponding to the lines connected to a given country, in which case they are set to one.

Figure~\ref{fig:countries-cap} focuses on the four countries with the highest added incumbent capacity (in TWkm) in the optimal solution.
In decreasing-capacity order, those are France, Germany, Great Britain and Norway.
For the two last ones, the necessary additional capacity converges to 0 TWkm for a suboptimality coefficient $\epsilon$ of less than 5\% even though they had around 20 TWkm of new installed capacity in the optimal solution. 
France and Germany start from an even higher capacity level in the optimal solution, with 58 and 34 TWkm of added incumbent capacity.
When the coefficient $\epsilon$ reaches 20\%, they still have a non-zero additional capacity, with respectively 5 and 8 TWkm, but the decrease is greater than 90\% for France and 75\% for Germany.\\

\noindent {\bf Individual lines.}
Necessary conditions can be used to identify critical lines in the network. In this case, only one value in $\mathbf{d}$ is set to 1.
Figure~\ref{fig:lines-cap} shows the value of the necessary conditions for the four lines with the largest capacity increase (in TWkm) in the optimal solution: ES-FR, FR-IT, FR-GB, and DE-FR.

The main conclusion drawn from this graph is that no individually line needs to be necessarily expanded to avoid a suboptimality greater than 2.5\%.

\subsection{Necessary conditions on storage capacity}

In the model, there is no pre-existing storage capacity and investment can be made at each bus.
In the optimal set-up, 249 GW of Li-ion batteries are built and store 98 TWh over the simulated year.
However, Figure~\ref{fig:storage-cap} shows the necessary condition reaches 0 GW for a suboptimality coefficient $\epsilon$ as small as 5\%.

\subsection{Necessary conditions on RES capacity}

Finally, investments in renewable energy sources, including onshore and offshore wind turbines and utility-scale PV panels, are analysed.
Four types of necessary conditions are computed: one per technology corresponding to the required minimum in new capacity for that technology and one for the required minimum in the sum of capacities in the three technologies. 
As for storage, a greenfield approach is used and the new capacity is equal to the total capacity that is installed.

Figure~\ref{fig:generation-cap} shows that investments in renewable energies are essential as the minimum capacity required to be $\epsilon$-optimal does not drop below 600 GW even as the suboptimality coefficient $\epsilon$ rises to 20\%. 
However, this is less clear for each RES technology individually.
While the minimum requirement for offshore wind stays consistent with increasing values of the suboptimality coefficient $\epsilon$, the necessary conditions for onshore wind and utility PV converge to 0 GW.\\

\section{Conclusion and future work} \label{sec:conclusion}

In this paper, a framework offering a change of focus for optimisation model methods was presented and applied for capacity expansion planning.
Deviating from cost-optimal focused studies, we advocate for the search of non-implied necessary conditions for $\epsilon$-optimality to inform decision-makers efficiently.\\

The concepts required to define this search in a generic case were formalised.
A methodology was then presented to derive necessary conditions in the specific context where conditions consist of constrained sums.
Finally, to illustrate the framework, this methodology was applied to a test case related to capacity expansion planning at the European level, focusing on the minimum investments in transmission, storage and generation required for $\epsilon$-optimality.\\

This work sets the ground for further developments of the presented framework.
First, the framework was specified for a fixed set of parameters that define the shape of the feasible space $\mathcal{X}$ and of the objective function $f$.
Changing the value of the parameters could thus impact $\mathcal{X}$ and $f$, and in turn, the optimal solution and the $\epsilon$-optimal spaces $\mathcal{X}^\epsilon$.
As a result, there is no guarantee that necessary conditions found for a fixed set of parameters would remain the same for a different set of parameters.
The concept of necessary conditions could thus be extended to overcome this limitation by defining sets of \textit{meta}-necessary conditions valid for different sets of parameters.
Determining such necessary conditions and providing guarantees on implications would require more advanced techniques than the one presented in this article.

Second, in this paper, we only presented an algorithmic solution for computing non-implied necessary conditions in the context where conditions consist of constrained sums of variables. 
It would be interesting to propose algorithmic solutions for other types of conditions.

Third, while we focused on cost-based $\epsilon$-optimality, this concept and the one of necessary conditions can naturally be extended to other objectives. 

Finally, it would be interesting to investigate whether other fields than capacity expansion planning could benefit from the framework introduced in this paper.\\

\bibliographystyle{IEEEtran}
\bibliography{pscc2022}

\appendix

\section{Network model} \label{app:network-model}

Modelling and optimising the network is done using PyPSA \cite{PyPSA} in conjunction with REplan\cite{replan}.

\subsection{Topology - Buses and Lines}

The initial topology of the network is based on the TYNDP18 2027 reference grid developped by ENTSO-E\cite{tyndp18}.
It consists of bi-directional net transfer capacities (NTC) between countries or regions inside countries.
To obtain a one-node-per-country topology, nodes are clustered per country, outgoing lines capacities are summed, and intra-country lines are removed. 

Connections are modelled as bi-directional links using a transportation model. 
The initial capacity of each line is set to the maximum of both NTCs. 
Lines crossing seas are considered to be HVDC cables, while other lines are represented as HVAC lines. 
For simulating the N-1 stability constraint, the maximum power flow across any line is set to 70\% of its installed capacity (as suggested in \cite{PyPSAEur}).

For keeping the expansion realistic, an upper bound is fixed on the maximum capacity per line. 
This upper bound is set based on the NTCs of the 'Global Climate Action 2040' scenario of TYNDP2018.
However, to provide some slack to the model, this capacity is multiplied by a pre-defined factor of 3. 
Note that for a multiplication factor equal to 2, around 5\% of the load was shed in the optimal solution.

\subsection{Load}

The model is solved at a 2-hourly resolution. 
At each time step, the load must be satisfied or, is shed for a cost of 3k\euro/MWh. 
Hourly load series per country are extracted from the Open Power System Data project\cite{OPSD}.
The reference year used in this model is 2018.

\subsection{Generation and Storage Technologies}

The model contains generation and storage technologies. For each technology, one representative plant is used per node where the pre-existing capacity or capacity expansion potential is not null. 
As detailed below, some of these technologies are expandable and others are not.\\

\noindent \textbf{Technologies with expandable capacity.}
Dispatchable capacity can be deployed in the form of CCGT and OCGT. They are the only technologies that produce CO2 emissions when generating electricity.

Short-term storage can be built as Li-Ion batteries. Those batteries are characterised by two elements: their peak power capacity and the maximum duration during which they can discharge this power. 
In the test case, this second element is fixed to 4 hours and multiplying by the peak power gives the storage capacity of the battery.

For these three technologies, no initial capacity and no upper limit on the amount of new capacity are considered.

Three types of renewable energy sources are added to the model: onshore wind generators, offshore wind generators and utility-scale PV power plants.
The per-country capacity factors profiles are obtained through Renewables.ninja, presented in \cite{PFENNINGER20161251, STAFFELL20161224}, while expansion potential are computed via GLAES \cite{Ryberg2018}. 
GLAES is parametrised such that, on a cumulative basis, a maximum of 447 GW can be built for onshore wind, 1077 GW for offshore wind and 1150 GW for PV. 
No initial capacity is considered. 
The energy produced by those generators can be curtailed without incurring any supplementary cost.\\

\noindent \textbf{Technologies with fixed capacity.}
New investments in nuclear power are not considered.
Generators in Belgium and Germany, and those commissioned before 1980, are removed from the model.
Using the JRC Open Power Plants 
database \cite{jrc-open}, this leads to 94GW of capacity which is in line with the projections made in the 2016 EU Reference Scenario \cite{eu2016}.

Hydro-power is modelled through the addition of pre-existing run-of-river generators with a capacity of 34 GW, reservoirs with 105 GW and pumped-hydro storage with 55 GW.
Capacities and locations around Europe are extracted from the JRC Hydro-power plants database\cite{jrc-hydro}.
For more information on the modelling of input flows, the interested reader can refer to the supplementary material of \cite{RADU2022117700}.

\subsection{Input parameters and data}
\balance

All experiments can be reproduced using the code which is available at \cite{pscc2022Code}.
The input data used for generating the results presented in this paper can be retrieved in \cite{pscc2022Data} and is preprocessed using the open-source tool EPIPPy \cite{epippy}. 
The repository also contains the output of the PyPSA runs and a document describing the techno-economic parameters used in the model and the sources from which they were determined. 

\end{document}